\documentclass[a4paper]{amsart}       
\usepackage{graphicx}
\usepackage{amsmath,amsfonts,amssymb,amsthm}
\usepackage[all]{xy}

\newcommand{\set}[1]{\left\{#1\right\}}
\newcommand{\abs}[1]{\left|#1\right|}

\newcommand{\CA}{\mathcal{A}}
\newcommand{\CB}{\mathcal{B}}
\newcommand{\CC}{\mathcal{C}}

\newcommand{\CF}{\mathcal{F}}
\newcommand{\CG}{\mathcal{G}}
\newcommand{\CM}{\mathcal{M}}
\newcommand{\CT}{\mathcal{T}}
\newcommand{\CU}{\mathcal{U}}
\newcommand{\CP}{\mathcal{P}}
\newcommand{\CQ}{\mathcal{Q}}
\newcommand{\CR}{\mathcal{R}}
\newcommand{\CS}{\mathcal{S}}

\newcommand{\CW}{\mathcal{W}}

\newcommand{\CAF}{\mathcal{A}_\mathcal{F}}

\newcommand{\fr}{\mathsf{fr}}

\newcommand{\Fo}{F{\o}lner }

\newcommand{\sd}{\bigtriangleup}
\newcommand{\eps}{\varepsilon}
\DeclareMathOperator{\Int}{Int}

\newtheorem{thm}{Theorem}[section]
\newtheorem*{thm*}{Theorem}
\newtheorem{lem}[thm]{Lemma}
\newtheorem*{lem*}{Lemma}

\newtheorem*{cor*}{Corollary}

\newtheorem*{prop*}{Proposition}
\theoremstyle{definition}
\newtheorem{defn}[thm]{Definition}
\newtheorem*{defn*}{Definition}
\theoremstyle{remark}

\newtheorem*{rem*}{Remark}
\newtheorem{fct}[thm]{Fact}
\newtheorem*{fct*}{Fact}

\newtheorem*{example*}{Example}

\newtheorem*{que*}{Question}

\begin{document}

\title{Zero-dimensional extensions of amenable group actions}


\author{Dawid Huczek}
\address{\hskip- \parindent
    Dawid Huczek, Faculty of Pure and Applied Mathematics, Wroclaw University of Science and Technology, Wybrze\.ze Wyspia\'nskiego 27, 50-370 Wroc\l aw, Poland}
\email{dawid.huczek@pwr.edu.pl}


\begin{abstract}
We prove that every dynamical system $X$ with free action of a countable amenable group $G$ by homeomorphisms has a zero-dimensional extension $Y$ which is faithful and principal, i.e. every $G$-invariant measure $\mu$ on $X$ has exactly one preimage $\nu$ on $Y$ and the conditional entropy of $\nu$ with respect to $X$ is zero. This is a version of the result in \cite{DH2} which establishes the existence of zero-dimensional principal and faithful extensions for general actions of the group of integers.
\end{abstract}
\maketitle

\section{Introduction}
\label{sec:intro}
In this paper we prove the existence of a zero-dimensional, principal and faithful extension for any dynamical systems with free action of an amenable group $G$. This is a version of the result proved in \cite{DH2} --- it is not strictly a strengthening, since while we do generalize theorem 3.1 of said paper for the case of any amenable group, we need to assume that the action of the group be free. The present paper is self-contained, but the reader may want to refer to the proof in \cite{DH2} or \cite{DH1}, where similar ideas are implemented in a simpler setting.

\section{Preliminaries}
\label{sec:prelim}
\subsection{Basic properties of amenable group actions}
A discrete, countable group $G$ is an $\emph{amenable group}$ if it contains a \Fo sequence, i.e. a sequence of finite sets $F_n \subset G$ such that for any $g\in G$  
\[\lim_{n \to \infty}\frac{\abs{F_n \sd gF_n}}{\abs{F_n}}=0,\]
where $gF=\set{gf:f\in F}$, $\abs{\cdot}$ denotes the cardinality of a set, and $\sd$ is the symmetric difference. The elements of the \Fo sequence will be called \emph{\Fo sets}. Without loss of generality (see \cite[Corollary 5.3]{N}) we can assume that the sets in the F\o lner sequence are symmetric (i.e. $F_n^{-1}=F_n$ for every $n$) and contain the neutral element of $G$.

When $F$ and $E$ are finite subsets of $G$ and $0<\delta<1$, we say that $F$ is \emph{$(E,\delta)$-invariant} if
\[\frac{\abs{F\sd EF}}{\abs{F}}\leq \delta,\]
where $EF=\set{gf:g \in E,f\in F}$. Note that if $E$ contains the neutral element of $G$ (which we will be denoting by $e$), then $(E,\delta)$-invariance is equivalent to a simpler condition:
\[\abs{EF}\leq (1+\delta)\abs{F}.\]

Also observe that if $(F_n)$ is the \Fo sequence, then for any finite $E \subset G$ and any $\delta>0$ there exists an $N$ such that for $n>N$ the set $F_n$ is $(E,\delta)$-invariant. 

\begin{defn}
    For a finite $T\subset G$ and $\eps\in [0,1]$ we say that $T'\subset T$ is a $(1-\eps)$-subset of $T$ if $\abs{T'}\geq(1-\eps)\abs{T}$.
\end{defn}

\begin{defn}
Let $K$ be a finite subset of $G$ and let $T\subset G$ be arbitrary. The \emph{$K$-core} of $T$, denoted as $T_K$, is the set $\set{g\in T: Kg\subset T}$ (this is the largest subset $T'\subset T$ satisfying $KT'\subset T$).
\end{defn}

\begin{lem}\label{lem:large_core}
For any $\eps>0$ and any finite $K\subset G$ there exists a $\delta$ (in fact $\delta = \frac\eps{|K|}$), such that if $T\subset G$ is finite and $(K,\delta)$-invariant then the $K$-core $T_K$ is a $(1-\eps)$-subset of $T$.
\end{lem}

\begin{proof} Note that $(K,\delta)$-invariance of $T$ implies that
$$(\forall g\in K)\ \ |gT\setminus T|<\delta|T|,$$
i.e.,
$$(\forall g\in K)\ \ |T\cap g^{-1}T| = |gT\cap T| >(1-\delta)|T|.
$$
This yields
$$|T_K|=\left|\bigcap_{g\in K}(T\cap g^{-1}T)\right|>(1-|K|\delta)|T|=(1-\eps)|T|.$$
\end{proof}
\begin{defn}
    Let $A$ and $B$ be subsets of $G$. We say that $A$ \emph{lies on the boundary} of $B$ if $A\cap B\neq\emptyset$ and $A\setminus B\neq\emptyset$. 
\end{defn}

\subsubsection{Lower Banach density}
\begin{defn}Let $H \subset G$. For any finite $F \subset G$ set
    \[D_F(H)=\inf_{g \in G}\frac{\abs{H \cap Fg}}{\abs{F}}.\]
    The value
    \[D(H)=\sup\set{D_F(H):F \subset G, \abs{F}<\infty}\]
    is called the \emph{lower Banach density} of $H$.
\end{defn}
\begin{lem}
    If $(F_n)$ is the \Fo sequence then for any $H \subset G$
    \[D(H)=\lim_{n \to \infty}D_{F_n}(H).\]
\end{lem}
\begin{proof}
    Let $\delta>0$ and let $F\subset G$ be a finite set such that $D_F(H)\ge D(H)-\delta$. Let $n$ be so large that $F_n$ is $(F,\delta)$-invariant. For $g\in G$ we have
    \[\abs{H\cap Ffg}\ge D_F(H)\abs{F}\]
    for any $f\in F_n$. Therefore there are at least $D_F(H)\abs{F}\abs{F_n}$ pairs $(f',f)$ such that $f'\in F,f\in F_n$ and $f'fg\in H$. This implies that for at least one $f'\in F$ we have
    \[\abs{H\cap f'F_ng}\ge D_F(H)\abs{F_n}.\]
    Since $F_n$ is $(F,\delta)$-invariant (and thus so is $F_ng$) and $f'\in F$, we have the inequality
    \[\abs{H\cap f'F_ng} \le \abs{H\cap FF_ng} \le \abs{H\cap F_ng}+\delta\abs{F_n},\]
    and therefore
    \[\abs{H\cap F_ng}\ge (D_F(H)-\delta)\abs{F_n}.\]
    We have shown that $D_{F_n}(H)\ge D_F(H)-\delta\ge D(H)-2\delta$ which concludes the proof.
\end{proof}
\subsection{Actions of amenable groups}
Let $X$ be a compact metric space and $G$ an amenable group. A \emph{continuous action} of $G$ on $X$ is a homomorphism $\Theta:G \to \text{Homeo}(X)$, i.e. a mapping such that $\Theta(gh)=\Theta(g)\circ\Theta(h)$. For the sake of concision, in case of a single, fixed group action, we do not write $\Theta(g)(x)$, but rather $g(x)$ or even $gx$. The set $X$ with the continuous action of $G$ is called a topological dynamical system.

The action of $G$ on $X$ is called \emph{free} if no point of $X$ is fixed by no element of $G$, i.e. if
\[\forall_{x\in X}\forall_{g \in G}\left(gx= x \implies g=e\right).\]

By $\CM(X)$ we will mean the set of all Borel probability measures on $X$, and by $\CM_{\Theta}(X)$ we mean the subset of $\CM(X)$ consisting of $\Theta$-invariant measures, i.e $\mu\in \CM_{\Theta}(X)$ if and only if for every Borel set $B\subset X$ and every $g\in G$ we have $\mu(g(B)))=\mu(B)$. 

If $(X,\Theta^{(X)})$ and $(Y,\Theta^{(Y)})$ are topological dynamical systems with the actions of the same group $G$, we say that $Y$ is an extension of $X$ (or, equivalently, that $X$ is a factor of $Y$), if there exists a continuous map $\pi$ from $Y$ onto $X$ which commutes with the actions, i.e. for any $g\in G$ we have$\pi(gy)=g(\pi(y))$. Such a map $\pi$ induces a continuous affine map from $\CM_{\Theta^{(Y)}}(Y)$ onto $\CM_{\Theta^{(X)}}(X)$, which we will also denote by $\pi$. We say that an extension is \emph{faithful} if the corresponding map between the simplices of invariant measures is injective, i.e. any invariant Borel probability measure on $X$ has exactly one preimage on $Y$.

\subsection{Symbolic and array systems}
Symbolic dynamical systems for group actions can be defined similarly to the case of a single action: Let $\Lambda$ be the \emph{alphabet}, that is a finite set with discrete topology. Let $Z=\Lambda^G$ and define the following action of $G$ on $Z$:
\[\forall_{g,h\in G}\forall_{x \in Z}\,\,gx(h)=x(hg).\]
Since $(g_1g_2)x(h)=x(hg_1g_2)=g_2x(hg_1)=g_1(g_2x)(h)$, this is indeed a continuous action of $G$ on $Z$. A \emph{symbolic dynamical system} is any closed set $X \subset Z$ that is invariant under the action of $G$ i.e. $g(X)= X$ for all $g \in G$.

Let $Q$ be a block over the alphabet $\Lambda$ defined by a finite set of coordinates $A \subset G$, i.e. a function from $A$ into $\Lambda$, and let $P$ be a block over the same alphabet determined by a finite set of coordinates $B$. The frequency of $Q$ in $P$ is defined as
\[\fr_P(Q)=\frac{1}{\abs{P}}\abs{g\in B:Ag\subset B \wedge P(Ag)=Q},\]

An array dynamical system is defined in a similar manner: let $\Lambda_k$ be a finite set with discrete topology and let $Y_\mathsf{arr}=\prod_{k=1}^\infty\Lambda_k^G$. The sets of $Y_\mathsf{arr}$ can be interpreted as arrays with ,,rows'' (or rather layers) indexed by elements of $G$ and columns indexed by natural numbers, i.e. arrays of the form $\set{y_{k,g}}_{k\ge 1,g\in G}$, where $y_{k,g} \in \Lambda_k$. With the action of the horizontal shift $\Psi$ (defined as $\Psi(g)(y)_{k,h}=y_{k,hg}$) the set $Y_\mathsf{arr}$ becomes a zero-dimensional dynamical system. An \emph{array dynamical system} is any closed, $\Psi$-invariant subset $Y$ of such a $Y_\mathsf{arr}$. Any action of an amenable group on a zero-dimensional system can be represented as an array system.

\subsection{Entropy}

We recall the basic definitions and facts of the entropy theory of amenable group actions. Let $X$ be a compact metric space, $\Theta$ be a continuous action of an amenable group $G$ and let $\mu \in \CM_\Theta(X)$. For any finite partition $\mathcal{A}$ of $X$ into measurable sets we define the entropy of a partition as
\[H(\mu,\mathcal{A})=-\sum_{A \in \mathcal{A}}\mu(A)\ln\mu(A),\]
as well as the expression
\[H_n(\mu,\mathcal{A})=\frac{1}{\abs{F_n}}H(\mu,\bigvee_{g\in F_n}g^{-1}(\mathcal{A})).\]

The sequence $H_n$ is known to converge to its infimum, which allows one to define
\[h(\mu,\mathcal{A})=\lim H_n(\mu,\mathcal{A}).\]
Finally the \emph{entropy of a measure} is given as
\[h(\mu)=\sup_{\mathcal{A}}h(\mu,\mathcal{A}),\]
where the supremum is taken over all finite partitions of $X$.

If $\mathcal{A}$ and $\mathcal{B}$ are two finite partitions of $X$, then we can define the conditional entropy of $\mathcal{A}$ with respect to $\mathcal{B}$ as
\[H(\mu,\mathcal{A}|\mathcal{B})=\sum_{B \in \mathcal{B}}\mu(B)H(\mu_B,\mathcal{A}),\]
where $\mu_B(A)\in \CM(X)$ is defined by $\mu_B(A)=\mu(A|B)$ for all Borel sets. Then we can proceed to define
\[H_n(\mu,\mathcal{A}|\mathcal{B})=\frac{1}{n}H(\mu,\bigvee_{g\in F_n}g^{-1}(\mathcal{A})|\bigvee_{g\in F_n}g^{-1}(\mathcal{B})).\]
Again, the sequence $H_n$ is known to converge to its infimum, which allows one to define
\[h(\mu,\mathcal{A}|\mathcal{B})=\lim H_n(\mu,\mathcal{A}|\mathcal{B}).\]
Suppose now that we have a system $(Y,\Psi)$ that is an extension of $(X,\Theta)$ by a map $\pi$. Let $\nu$ be an invariant measure on $Y$. Any partition $\mathcal{B}$ of $X$ can be lifted to a partition $\pi^{-1}\mathcal{B}$ of $Y$. Define
\[h(\nu,\mathcal{A}|X)=\inf_{\mathcal{B}} h(\nu,\mathcal{A}|\pi^{-1}\mathcal{B}),\]
where the infimum is taken over all finite partitions of $X$. Finally, define
\[h(\nu|X)=\sup_{\mathcal{A}}h(\nu,\mathcal{A}|X),\]
where once again the supremum is taken over all finite partitions of $Y$. If the image $\pi\nu$ of $\nu$ by the factor map $\pi$ has finite entropy, then it is not difficult to see that $h(\nu|X)=h(\nu)-h(\pi\nu)$.

We will make use of the following fact.
\begin{fct}
    Let $Y$ be an array system and let $X$ be a factor of $Y$. Let $\mathcal{R}_k$ be the partition of $Y$ defined by cylinders of height $k$ determined by the coordinate $e$. Then for any measure $\nu$ on $Y$ we have $h(\nu|X)=\lim_k h(\nu,\mathcal{R}_k|X)$.
\end{fct}
To see that it is so, it suffices to observe two facts. Firstly, that the family $\set{\mathcal{R}_k}$ together with its images under the group elements generates the Borel $\sigma$-algebra on $Y$. Secondly, if $l<k$ then $\CR_l\prec \CR_k$, and therefore $h(\nu,\CR_l|X)<h(\nu,\CR_k|X)$.

We recall a key definition:
\begin{defn}Suppose a dynamical system $Y$ an extension of the system $X$ via the map $\pi$. $Y$ is a \emph{principal extension} if $h(\nu|X)=0$ for every invariant measure $\nu$ on $Y$.
\end{defn}
If $X$ has finite topological entropy, then by the variational principle $\pi\nu$ has finite entropy for each invariant measure $\nu$ on $Y$ so the extension is principal if and only if $h(\nu)=h(\pi\nu)$ for each $\nu$. This implies that $Y$ has the same topological entropy as $X$ (this holds also in case of infinite entropy).
\subsection{Continuity of the entropy functions}
In the main proof we will consider entropy as a function of the measure, and we will need several basic facts about the continuity of this function. First of all: 
\begin{fct}The function $\mu \mapsto \mu(A)$ on $\CM(X)$ is upper semicontinuous if $A$ is closed and lower semicontinuous if $A$ is open.
\end{fct} Since $\mu(\Int(A))\leq \mu(A) \leq \mu(\overline{A})$ and the three are equal if the boundary of $A$ has measure $0$, we have the following:
\begin{fct}The function $\mu \mapsto \mu(A)$ on $\CM(X)$ is continuous at every $\mu$ such that $\mu(\partial A)=0$.
\end{fct}
Using the fact that the limit defining $h(\mu,\CA|\CB)$ is also the infimum, we easily arrive at the following:
\begin{fct}For any finite partitions $\CA,\CB$ of $X$ the function $\mu \mapsto h(\mu,\CA|\CB)$ on $\CM_\Theta(X)$ is upper semicontinuous at every $\mu$ such that $\mu(\partial A)=0$ for every $A\in \CA$ and $\mu(\partial B)=0$ for every $B \in \CB$.
\end{fct}
Finally:
\begin{fct}\label{fct:1.9}If a dynamical system $Y$ with action $\Theta$ of an amenable group $Y$ is an extension of $X$ and $\CA$ is a finite partition of $Y$, then the function $\mu \mapsto h(\mu,\CA|X)$ on $\CM_\Theta(Y)$ is upper semicontinuous at every $\mu$ such that $\mu(\partial A)=0$ for every $A\in \CA$.
\end{fct}
To observe that, note that $h(\mu,\CA|X)$ is the infimum of $h(\mu,\CA|\CB_n)$, provided that the diameter of the largest set in $\CB_n$ tends to $0$. Since we can construct partitions into sets of arbitrarily small diameter that all have boundaries whose measure $\mu$ is $0$, Fact \ref{fct:1.9} now follows.

\section{Principal and faithful extensions of amenable group actions}\label{sec:faithful_group}

The goal of this section is to prove the following theorem which is a version of theorem 3.1 of \cite{DH2} for the case of free actions of amenable groups:
\begin{thm}\label{thm:faithful_group}
    Let $(X,\Theta^{(X)})$ be a topological dynamical system, where $\Theta^{(X)}$ is a free action of an amenable group $G$ by homeomorphisms on a compact metric space $X$. There exists a compact, zero-dimensional metric space $Y$ with free action $\Theta^{(Y)}$ of $G$ such that the system $(Y,\Theta^{(Y)})$ is a principal and faithful extension of $(X,\Theta^{(X)})$.
\end{thm}

For groups that are not free, we can obtain a slightly weaker result (as an easy corollary of the above), where the extension is principal, but not necessarily faithful:

\begin{thm}\label{thm:principal_group}
    Let $(X,\Theta^{(X)})$ be a topological dynamical system, where $\Theta^{(X)}$ is an action of an amenable group $G$ by homeomorphisms on a compact metric space $X$. There exists a compact, zero-dimensional metric space $Y$ with free action $\Theta^{(Y)}$ of $G$ such that the system $(Y,\Theta^{(Y)})$ is a principal extension of $(X,\Theta^{(X)})$.
\end{thm}

Before we can prove theorem \ref{thm:faithful_group}, we need to introduce several tools that allow us to perform and analyze block operations in array systems.

\subsection{Quasitilings of amenable groups}

\begin{defn} A \emph{quasitiling} $\CT$ of a countable amenable group $G$ is determined by two objects:
    \begin{enumerate}
        \item a finite collection $\CS(\CT)$ of finite subsets of $G$ containing the unit $e$,
        called \emph{the shapes}.
        \item a finite collection $\CC(\CT) = \{\CC(S):S\in\CS(\CT)\}$ of disjoint subsets of $G$, called \emph{center sets} (for the shapes).
    \end{enumerate}
    The quasitiling is then the family $\CT=\{Sc:S\in\CS(\CT),c\in \CC(S)\}$. We require that every $T\in \CT$ have a unique representation $T=Sc$ for some $S\in \CS(\CT), c\in \CC(S)$.\footnote{This requirement is stronger than asking that different tiles have different centers. Two tiles $Sc$ and $S'c'$ may be equal even though $c\neq c'$
        (this is even possible when $S=S'$). However, when the tiles are disjoint, then the (stronger) requirement
        follows automatically from the fact that the centers belong to the tiles.} Hence, by the \emph{tiles} of $\CT$ (denoted by the letter $T$) we will mean either the sets $Sc$ or the pairs $(S,c)$ (i.e., the tiles with defined centers), depending on the context. 
\end{defn}
Note that every quasitiling $\CT$ can be represented in a symbolic form, as a point $x_\CT\in\Lambda^G$, with the alphabet $\Lambda = \CS(\CT)\cup\{\emptyset\}$, as follows: $x_\CT(g)=S$ if $g\in \CC(S)$ (the uniqueness condition on tiles implies that every $g$ will belong to at most one $\CC(S)$), and $x_\CT(g)=\emptyset$ otherwise.

\begin{defn}\label{qt} Let $\eps\in[0,1)$ and $\alpha\in(0,1]$. A quasitiling $\CT$ is called
    \begin{enumerate}
        \item \emph{$\eps$-disjoint} if there exists a mapping $T\mapsto T^\circ$ ($T\in\CT$) such that
        \begin{itemize}
            \item $T^\circ$ is a $(1-\eps)$-subset of $T$, and
            \item $T\neq T'\implies T^\circ\cap {T'}^\circ=\emptyset$;
        \end{itemize}
        \item \emph{disjoint} if the tiles of $\CT$ are pairwise disjoint;
        \item \emph{$\alpha$-covering} if $D(\bigcup\CT)\ge\alpha)$ (by $\bigcup\CT$ we mean the union of all tiles of $\CT$).
    \end{enumerate}
\end{defn}

\begin{defn}
    If $X$ is a zero-dimensional compact metric space and $G$ is an amenable 
    group acting on $X$ by homeomorphisms, then a \emph{dynamical quastiling} $\CT$ (or $\CT(\cdot)$)
    is a map $x\mapsto \CT(x)$ which assigns to every $x\in X$ a quasitiling of 
    $G$ such that the set of all shapes $\CS(\CT) =\bigcup_{x\in X}\CS(\CT(x))$ is 
    finite, and $x\mapsto \CT(x)$ is a factor map from $X$ onto a symbolic 
    dynamical system over the alphabet $\Lambda=\CS\cup\set{\emptyset}$. We say 
    that a dynamical quasitiling is $\eps$-disjoint, disjoint or 
    $\alpha$-covering, if $\CT(x)$ has the respective property for every $x$.
\end{defn}

In this paper we will use the following theorems (\cite{DH3}):

\begin{thm}\label{disjoint_tilings}
    Let $X$ be a compact, zero-dimensional metric space and let $G$ be a 
    countable amenable group acting freely on $X$ by homeomorphisms, with a F\o 
    lner sequence $(F_n)$ of symmetric sets containing the unit. Given $\eps>0$ 
    and any positive integer $n_0$, there exists a dynamical quasitiling 
    $x\mapsto \CT(x)$ which is \emph{disjoint}, and $(1-\eps)$-covering, and 
    such that every shape $S$ of every $\CT(x)$ is a $(1-\eps)$-subset of some 
    F\o lner set $F_{n(S)}$ where $n(S)>n_0$.
\end{thm}

\begin{thm}\label{congruent_tilings}
    Let $G$ be an amenable group acting freely on a zero-dimensional metric space $X$ and let $x\mapsto \CT(x)$ be any disjoint dynamical quasitiling of $G$. For any $\eps>0$, any finite $K\subset G$ and any $\delta>0$ there exists a disjoint, 
    $(1-\eps)$-covering dynamical quasitiling $x\mapsto \CT'(x)$ such that every shape of $\CT'(x)$ is $(K,\delta)$-invariant, and every tile of $\CT(x)$ is either a subset of some tile of $\CT'(x)$ or is disjoint from all such tiles.
\end{thm}

\begin{lem}\label{lem:disjoint_quasitiling}
    Let $Y$ be a zero-dimensional dynamical system with free action of an amenable group $G$. For any $\eps>0$ and a finite set $A\subset G$ there exists a disjoint, dynamical quasitiling $\CT$, such that replacing all tiles of $\CT$ by their $A$-cores yields a $(1-\eps)$-covering quasitiling. Furthermore, the tiles of $\CT$ can be assumed to be $(F,\gamma)$-invariant for any previously fixed $F$ and $\gamma$.
\end{lem}
Intuitively this lemma means that the quasitiling $\CT$ is good enough to yield a well-covering quasitiling even after replacing the tiles with their $A$-cores.
\begin{proof}
    There exists a $\gamma'$ such that if $T$ is $(F,\gamma')$-invariant and $T'$ is a $(1-\gamma')$-subset of $T$, then $T'$ is $(F,\gamma)$-invariant. In addition there exists a $\delta$ such that if $T$ is $(A,\delta)-$invariant then $T_A$ is a $(1-\gamma')$-subset of $T$. Decreasing $\gamma'$ if necessary, we can assume that $(1-\gamma')^2>1-\eps$. Finally, there exists a $\gamma''<\gamma'$ and $n_0$ such that any $(1-\gamma'')$-subset of any F\o lner set $F_n$ for $n\geq n_0$ is both $(F,\gamma')$ and $(A,\delta)$-invariant. 
    
    Apply theorem \ref{disjoint_tilings} for $\gamma''$ and $n_0$. This yields a disjoint, dynamical, $(1-\gamma'')$-covering quasitiling $\CT'$ whose tiles are all $(1-\gamma'')$ - subsets of F\o lner sets $F_n$ for $n\geq n_0$. This means that every tile $T$ of $\CT$ is $(A,\delta)$-invariant, and thus $T_A$ is a $(1-\frac{\gamma'}{2})$-subset of $T$. This in turn means that it's a $(1-\gamma')$-subset of $F_n$ for some $n\geq n_0$, and this gives us the $(F,\delta)$-invariance. In addition, since $\CT$ was $(1-\gamma')$-covering, then replacing all tiles by their $(1-\gamma')$-subsets will yield a quasitiling which is $(1-\gamma')^2$-covering (see e.g. lemma 3.4 in \cite{DHZ}), and thus $(1-\eps)$-covering.
    %
\end{proof}

\begin{lem}\label{lem:frequency}
    
    Let $Y$ be an array system and let $y\in Y$. Let $\eps>0$. For any finite $A\subset G$ there exists a $\delta>0$ with the following property: If $\CT$ is a $(1-\frac{\eps}{3})$-covering quasitiling by $(A,\delta)$-invariant tiles and $F$ is such that the union of tiles of $\CT$ contained in $F$ is a $(1-\frac{2\eps}{3})$-subset of $F$, then for every block $Q$ with domain $A$ the frequency of $Q$ in $Y(F)$ differs by at most $\eps$ from the average frequency of $Q$ in the blocks determined by the tiles of $\CT$ contained in $F$.
    
    %
    %
    %
    %
    %
\end{lem}
\begin{proof}
    Choose $\delta$ from lemma \ref{lem:large_core} for $A$ and $\frac{\eps}{3}$. 
    The frequency of $Q$ in $y(F)$ is a weighted average of its average frequency in the tiles of $\CT$ that are subsets of $F$ and its average frequency in their complement and on their boundaries. The latter part is not controlled (it is some number from the interval $[0,1]$), but its coefficient in the weighted average is at most $\frac{2\eps}{3}+\frac{\eps}{3}=\eps$, which concludes the proof.
\end{proof}
Note that the set $F$ in the above lemma can be, for instance, a sufficiently large \Fo set, as implied by the following lemma:
\begin{lem}\label{lem:small_boundary}
    Let $\CT$ be a quasitiling of $G$ and let $S$ be the union of all shapes of $\CT$. For any $\eps$ there exists a $\delta$ such that if $F$ is $SS^{-1},\delta)$-invariant, then the union $E$ of all the tiles of $\CT$ that lie on the boundary of $F$ satisfies the inequality $\abs{E'\cap F}<\eps\abs{F}$.
\end{lem}
\begin{proof}
    If an element $f\in F$ belongs for some $i$ to a tile $Sc$ that lies on the boundary of $F$, then $c\in S^{-1}f$, and thus $SS^{-1}f$ contains $Sc$, therefore it is not a subset of $F$. This means $f$ does not belong to $F_{SS^{-1}}$, and the number of such $f$ is at most $\eps\abs{F}$, which concludes the proof. 
\end{proof}

\subsection{Constructing the extension}

Similarly as in the proof of Theorem 3.1 of \cite{DH2}, we need to establish the existence and certain properties of an action disjoint from an invariant measure. The author thanks Benjamin Weiss for supplying the crucial ideas:

\begin{lem}\label{lem:disjoint_group}Let $(X,\Theta^{(X)},\mu)$ be a measurable dynamical system with free actions of an amenable group $G$. There exists an ergodic action $\Theta^{(I)}_\mu$ of $G$ on $I=[0,1]$, such that the measurable dynamical systems $(I,\Theta^{(I)}_\mu,\lambda)$ and $(X,\Theta^{(X)},\mu)$ are disjoint, where $\lambda$ is the Lebesgue measure on $I$.
\end{lem}
\begin{proof}If $\Theta^{(X)}$ has entropy zero (for $\mu$), then any Bernoulli action is disjoint from it. Otherwise we can apply theorem 21 of \cite{FW}, which states that the set of actions disjoint with a given action is a $G_\delta$. Since this set is an equivalence class of the isomorphism relation, it is dense if it is nonempty. By the proposition preceding that theorem, the set of actions of entropy zero is a dense $G_\delta$. Let $\Theta^{(X)}_P$ be the Pinkser factor of $\Theta^{(X)}$. This factor has entropy zero, so there is an action disjoint from it (any Bernoulli action), therefore the set of actions disjoint from $\Theta^{(X)}_P$, as well as the set of zero entropy actions, are dense $G_\delta$ sets. The set of ergodic actions of $G$ on the interval is also a dense $G_\delta$ (see \cite{GK}), so there exists an ergodic, zero-entropy action $\Theta^{(I)}_\mu$ that is disjoint from $\Theta^{(X)}_P$. The actions $\Theta^{(X)}$ and $\Theta^{(X)}_P\times \Theta^{(I)}_\mu$ are both extensions of $\Theta^{(X)}_P$, where the former is relatively c.p.e, and the latter is a relatively zero entropy extension. By theorem 1 of \cite{GTW}, the actions $\Theta^{(X)}$ and $\Theta^{(X)}_P \times \Theta^{(I)}_\mu$ are relatively disjoint over $\Theta^{(X)}_P$, but that implies that $\Theta^{(X)}$ and $\Theta^{(I)}_\mu$ are also disjoint.
\end{proof}

For any $F\subset G$ and $\mu\in\CM(X)$ let $\mathbf{A}^\Theta_F(\mu)(B)=\frac{1}{\abs{F}}\sum_{g\in F}\mu(g^{-1}B)$.

\begin{lem}\label{lem:products_group}
    Let $(X,\Theta^{(X)})$ be a topological dynamical system with free action $\Theta^{(X)}$ of an amenable group $G$. Let $\mu\in \CM_{\Theta^{(X)}}(X)$, let $U$ be a neighborhood of $\mu \times \lambda$ in $\CM(X \times I)$, let $\Theta^{I}_\mu$ be an ergodic action of $G$ on $I$, such that the systems $(I,\Theta^{I}_\mu,\lambda)$ and $(X,\Theta^{X},\mu)$ are disjoint, and let $t$ be a generic point of $\lambda$. There exists a neighborhood $U_\mu$ of $\mu$ in $\CM(X)$ and a number $N_\mu$, such that for $n>N_\mu$ if $\mathbf{A}^{\Theta^{(X)}}_{F_n}(\delta_{x})\in U_\mu$, then $\mathbf{A}^{\Theta^{(X)}\times\Theta^{(I)}_\mu}_{F_n}(\delta_{(x,t)})\in \CU$.
\end{lem}
\begin{proof}
    Assume the lemma does not hold. Then there exist a sequence $(x_k)$ such that $\mathbf{A}^{\Theta^{(X)}}_{F_{n_k}}(\delta_{x_k})\to \mu$, but the measures $\mathbf{A}^{\Theta^{(X)}\times\Theta^{(I)}_\mu}_{F_{n_k}}(\delta_{(x_k,t)})$ all lie outside $U$. Passing to a subsequence, we can assume that the sequence $\mathbf{A}^{\Theta^{(X)}\times\Theta^{(I)}_\mu}_{F_{n_k}}(\delta_{(x_k,t)})$ converges. By the properties of the \Fo sequence the limit measure must be invariant. Furthermore, its marginal on $X$ is $\mu$ (due to the assumption on averages for $x_k$), and its marginal on $I$ is $\lambda$ (since $t$ is generic for $\lambda$). The only such measure is $\mu \times \lambda$, but the measures $\mathbf{A}^{\Theta^{(X)}\times\Theta^{(I)}_\mu}_{F_{n_k}}(\delta_{(x_k,t)})$ are all outside its neighborhood $U$, which is a contradiction. 
\end{proof}

We now have all the tools needed to prove the main theorem of this paper.
\begin{proof}[Proof of Theorem \ref{thm:faithful_group}]
    Let $(F_n)$ be the \Fo sequence of $G$ and $e$ the neutral element. Recall that we assume that $e \in F_n$ for all $n$, and that each $F_n$ is symmetric.

    Let $I=[0,1]$. Any continuous function $f:X \to [0,1]$ induces a partition $\mathcal{A}_f$ of $X \times I$ into two sets: $\set{(x,t):0 \leq t <f(x)}$ and $\set{(x,t):f(x)\leq t <1}$ (i.e. the sets of points below and above the graph of $f$). For a family $\CF$ of continuous functions we denote by $\CAF$ the partition $\bigvee_{f \in \CF}\CA_f$. Two useful observations are that $\CA_{\CF \cup \CG}=\CA_\CF \vee \CA_\CG$ and that $\CF \subset \CG$ implies $\CA_\CF \prec \CA_\CG$.
    
    For every $j$ let $\CF_j$ be a family of continuous functions from $X$ into $I$ and let $\CA_j$ be the partition of $X \times I$ generated by $\CF_j$. By adding the appropriate functions to the relevant families, we can assume that $\bigvee_{g \in F_{j-1}}g^{-1}\CA_{j-1} \prec \CA_j$ (this condition, similarly to the proof of Theorem 3.1 in \cite{DH2}, ensures that if we treat the elements of these partitions as symbols in consecutive alphabets, then any symbol in a sufficiently far row of the array system determines the symbols in preceding rows of the same column, even if the preceding rows have been modified by a continuous transformation). Let $\zeta_j$ be the diameter of the largest set of $\CA_j$. We will assume that $2\zeta_{j+1}<\zeta_j$.

    Consider the set of all arrays $y=y_{j,g}$ $(j \ge 1,g \in G)$ such that $y_{j,g} \in \CA_j$ (as mentioned above, we treat the elements of the partitions as symbols in the alphabets) and define a set $K_{j,g}(y)=\set{x \in X: d(x,\pi^{(1)}(y_{j,g}))\leq\zeta_j}$. An array $y$ is said to satisfy the \emph{column condition} if $K_{j+1,g}(y)\subset K_{j,g}(y)$ for all $j$ and $g$. It is easy to see that if an array $y$ satisfies the column condition, then for every $g$ there exists exactly one point
    \[x_g(y)=\bigcap_{j=1}^{\infty}K_{j,g}(y).\]
    
    Let $Y_\mathsf{col}$ be the set of all arrays $y$ such that $y$ satisfies the column condition and $x_{g}(y)=g(x_e(y))$. Let $\Theta^{(Y)}$ be the standard horizontal shift action on $Y_\mathsf{col}$ i.e. let $(gy)_{j,h}=y_{j,hg}$. Define $\pi_X(y)=x_e(y)$. Observe that $\pi_X$ commutes with the action of $G$: for any $y$ and $g$ we have
    \[\pi_X(gy)=x_e(gy)=x_g(y)=g(x_e(y))=g\pi_X(y).\]
    Therefore $Y_\mathsf{col}$ is an extension of $X$ by the map $\pi_X$. 
    
    We will construct countably many dynamical systems which will all be subsystems of $Y_{\mathsf{col}}$. When discussing sets of $\Theta^{(Y)}$-invariant measures, we will treat them as subsystems of  $\CM_{\Theta^{(Y)}}(Y_{\mathsf{col}})$ (which in particular gives meaning to the notion of proximity between measures on different systems of this class). 
    Fix some sequence $(\eps_k)$ of positive numbers decreasing to $0$. The systems $Y_k\subset Y_\mathsf{col}$ will be constructed in such a way that $\CM_{\Theta^{(Y)}}(Y_k)\subset \CU_k$, where $\CU_k$ will be a sequence of subsets of $\CM_{\Theta^{(Y)}}(Y_{\mathsf{col}})$ satisfying the following conditions:
    \renewcommand{\theenumi}{U\arabic{enumi}}
    \begin{enumerate}
        \item $\CU_{k+1}\subset{\CU_k}$.
        \item For any $k>0$ and any measure $\nu \in \CU_k$ we have $h(\nu,\mathcal{R}_{k}|X)<\eps_k$, (where $\mathcal{R}_{k}$ is the partition defined by cylinders of height $k$ determined by coordinate $e$).
        \item For any $k>0$ and any two measures $\nu_1,\nu_2 \in \CU_k$ the condition $\pi_X(\nu_1)=\pi_X(\nu_2)$ implies that $d^*(\nu_1,\nu_2)<\eps_k$, where $d^*$ is a metric on $\CM_{\Theta^{(Y)}}(Y_\mathsf{col})$ consistent with the weak-star topology.
    \end{enumerate}
    \renewcommand{\theenumi}{\arabic{enumi}}

    To begin with, let $Y_0$ be the closure of the set of array-names of points in $X \times I$ under the action of $\Theta^{(X\times I)}$ with respect to the partitions $\CA_j$. In other words, $Y_0$ is the closure of the set of all points $y \in Y_\mathsf{col}$ such that for some pair $(x,t) \in X \times I$ and for any $j$ and $g$ we have $(\Theta^{(X)}(g)(x) \in y_{j,g}$. 
    
    By a standard argument, $Y_0$ is an extension of $X \times I$ (we will denote the corresponding map by $\pi_0$) as well as of $X$ itself and the following diagram commutes:
    \[
    \xymatrix{
        &Y_0\ar[ld]_{\pi_0}\ar[dd]^{\pi_X}\\
        X\times I\ar[rd]^{\pi^{(1)}}&\\
        &X
    }
    \]
   Let the set $\CU_0$ be all of $\CM_{\Theta^{(Y)}}(Y_\mathsf{col})$ (all our requirements on the properties of $\CU_k$ only apply to the case $k>0$). 
    
    There are two important observations to be made here: Firstly, the only points in $X \times I$ that have multiple preimages under $\pi_0$ are the ones whose orbits enter the graph of a function from some $\CF_j$ (because graphs of continuous functions are closed). The product measure $\mu \times \lambda$ of the graph of any function is $0$ (recall that $\lambda$ denotes the Lebesgue measure on the interval). Therefore whenever $\nu$ is a measure on $Y_0$ that factors onto a measure $\mu \times \lambda$ on $X \times I$, then the set of points in $X \times I$ with multiple preimages by $\pi_0$ has zero measure $\mu\times\lambda$. This implies that the measure-theoretic systems $(Y_0,\Theta^{(Y)},\nu)$ and $(X \times I,\Theta^{(X \times I)},\mu \times \lambda)$ are isomorphic and $\nu$ is a unique preimage of $\mu\times\lambda$. Secondly, any block determined by columns from $F_n$ in rows from $1$ to $j$ in $Y_0$ is associated with a unique cell of $\bigvee_{g\in F_n}\Theta^{(X\times I)}(g^{-1})(\CA_j)$, the closure of which is the image (by $\pi_0$) of this block. 
    
    We will now proceed to create the systems $Y_k$, requiring them to have the following properties:
    \renewcommand{\theenumi}{Y\arabic{enumi}}
    \begin{enumerate}
        \item For each $k$, $\CM_{\Theta^{(Y)}}(Y_k) \subset \CU_k$.
        \item For each $k$, $Y_k=\Phi_k(Y_{k-1})$, where $\Phi_k$ is a conjugacy, and there exists an increasing sequence $j_k$ such that $\Phi_k$ leaves the rows with indices greater than or equal to $j_{k}$ unchanged.
    \end{enumerate}
    \renewcommand{\theenumi}{\arabic{enumi}}
    
    Observe that the property Y2 ensures that the diagram 
    \[
    \xymatrix{
        &Y_{0}\ar@{<->}[r]^{
            \Phi_0}\ar[ld]_{\pi_0}\ar[dd]^<<<<<<{\pi_X}&Y_1\ar[ldd]^<<<<<<{\pi_X}\ar@{<->}[r]^{\Phi_1}&Y_2\ar[lldd]^<<<<<<{\pi_X}\ar@{<->}[r]^{\Phi_2}&\cdots\\
        X\times I\ar[rd]^{\pi^{(1)}}\\
        &X\\
    }
    \]
    commutes and that for $j\geq j_k$ we still have the one-to-one correspondence between rectangles of size $j\times n$ in $Y_k$ and the cells of $\bigvee_{g\in F_n}\Theta^{(X\times I)}(g^{-1})(\CA_j)$, since this correspondence depends only on the contents of row $j$. Throughout, $\pi_k$ will denote the factor map from $Y_k$ onto $X\times I$ defined by composing the factorization $\pi_0$ of $Y_0$ with the conjugacy between $Y_0$ and $Y_k$.
    
    With every $Y_k$ we will associate three quasitilings and a certain set from the \Fo sequence, which we will denote by $H_k$. Two of the quasitilings, $\CT_k$ and $\CT_k^*$, will be dynamical quasitilings used to construct the next system via a block code (the tiles of $\CT^*_k$ are an analogue of the $k$-rectangles from the proof of theorem 3.1 of \cite{DH2}, whereas the tiles of $\CT_k$ are an analogue of the subrectangles between successive jump points). The third quasitiling, $\CT_k^{+}$, is static (i.e. it's a single quasitiling rather than a set of quasitilings corresponding to individual elements of $Y_k$) and is used for verifying that all invariant measures on $Y_k$ and subsequent systems are in $\CU_k$. The tiles of $\CT^+_k$ will be so large that all ,,test'' blocks of order $k$ will appear in it with controlled frequencies in any point of $Y_k$, while $H_k$ will have the same property even for subsequent systems. In addition, let $S_k,S_k^*,$ and $S_k^+$ denote sets which are unions of all the shapes of the respective quasitiling.

    We begin with the system $Y_0$, for which we set $\CU_0=\CM_{\Theta^{(Y)}}(Y_\mathsf{col})$ (which makes all conditions on invariant measures satisfied trivially), $j_0=1$, and $\CT_0,\CT_0^*, \CT_0^+$ are all trivial tilings of $G$ by single-element tiles. Also, let $H_0=F_1$.

    We will assume (and will guarantee by our inductive construction) that after the system $Y_{k-1}$ and the corresponding quasitilings have been constructed, then for any $l<{k-1}$ every tile of $\CT_l$ intersects at most one tile of $\CT_{k-1}$. 
    
    \medskip
    Let $\CP_k$ be the set of all $\Theta^{(Y)}$-invariant measures on $Y_{k-1}$, such that if $\nu \in \CP_k$, then $\pi_{k-1}(\nu)=\mu \times \lambda$ for some $\mu \in \CM_{\Theta^{(X)}}(X)$ (where $\lambda$ is the Lebesgue measure  on $I$). By the remarks we have made previously, we know that for any $\nu \in \CP_k$ the map $\pi_{k-1}$ is an isomorphism between $(Y_{k-1},\Theta^{(Y)},\nu)$ and $(X \times I,\Theta^{(X \times I)},\mu \times \lambda)$, therefore $h(\nu|X)=h(\mu \times \lambda|X)=0$, and any $\mu \in \CM_{\Theta^{(X)}}(X)$ has exactly one preimage in $\CP_k$. Since the function $h(\cdot|X)=0$ is zero on $\CP_k$ then so is the upper semicontinuous function $h(\cdot,\CR_k|X)$. This implies the existence of a convex, open set $\CU_k\supset \CP_k$ satisfying the conditions U1--U3 (each of them is satisfied on an open set containing $\CP_k$, so we can take $\CU_k$ to be a convex neighborhood of $\CP_k$ contained in the intersection of such sets).

    We want to construct $Y_k$ in such a way that $\CM_{\Theta^{(Y)}}(Y_k) \subset \CU_k$. There exist some $\delta_k$, $j_k$ and $d_k$, such that if two measures differ by less than $\delta_k$ on all cylinders determined by coordinates $(j,g)$ for $j<j_k$, $g \in F_{d_k}$, and one of those measures is in $\CP_k$, then the other must be in $\CU_k$. Denote the family of such ,,test'' blocks by $\CQ_k$. By the same reasoning as in theorem 3.1 of \cite{DH2} (since we assumed that $\bigvee_{g \in F_{j-1}}g^{-1}\CA_{j-1} \prec \CA_j$), we can assume that $j_k$ is so large that for any $y\in Y_{k-1}$ and any $g\in G$ the symbol $y_{j_k,g}$ determines all the $y_{j,g}$'s for $j\leq j_k$. 
     
    We will now formulate one final inductive assumption. It states that $Y_{k-1}$ has the following property Y3($l$) for all $l<{k-1}$:
    \renewcommand{\theenumi}{Y3($l$)}
    \begin{enumerate}
        \item For any $y\in Y_{k-1}$ there exists some $\nu_{y,l} \in \CP_l$, such that for all $Q \in \CQ_l$ we have $\abs{\fr_{y(H_l)}(Q)-\nu_{y,l}(Q)}<\delta_l$.
    \end{enumerate}
    \renewcommand{\theenumi}{\arabic{enumi}}

We begin by constructing the quasitiling $\CT_k$. Define the set \[E=S_1S^{-1}_1S_2S^{-1}_2\cdots S_{k-1}S^{-1}_{k-1}.\] Apply lemma \ref{lem:disjoint_quasitiling} with the constant $\eta_k$ (which depends only on $\delta_k$ and will be specified later) and the set $H_{k-1}^2EF_{d_k}$, obtaining a disjoint, dynamical quasitiling $\hat{\CT}_k$, whose $H_{k-1}^2EF_{d_k}$-cores are a $(1\!-\!\eta_k)$-covering family. For every tile $\hat{T}$ of $\hat{\CT}_k$ define $\widetilde{T}=(\hat{T})_{H_{k-1}^2E}$. This yields a disjoint, dynamial quasitiling $\widetilde{\CT}_k$, which is $(1\!-\!\eta_k)$-covering. 

Now, for every $x$ and every tile $\widetilde{T}$ of $\widetilde{\CT}_k(x)$, add to $\widetilde{T}$  all the tiles of $\CT^+_{k-1}$ that have nonempty intersection with it. Obviously, no tile of $\CT^+_{k-1}$ will lie on the boundary of such an enlarged set. Observe that this enlarged set is contained in $S_{k-1}S^{-1}_{k-1}\widetilde{T}_k$, therefore also in $\hat{T}_k$. Now add to the obtained set all the tiles of $\CT^+_{k-2}$ that have nonempty intersection with it. Again, observe that no tile of either $\CT^+_{k-2}$ or $\CT^+_{k-1}$ (thanks to the first inductive assumption on the tilings) lies on the boundary of this enlarged set. Also, the new set is a subset of $S_{k-2}S^{-1}_{k-2}S_{k-1}S^{-1}_{k-1}\widetilde{T}_k$. By continuing this procedure for $\CT^+_{k-3},\ldots,\CT^+_1$, we obtain a set $T_k\supset \widetilde{T}_k$ with the following properties: $T_k$ is a subset of $E\widetilde{T}_k$ (and thus also of $(\hat{T}_k)_{H_{k-1}^2}$) and for every $l<k$ no tile of $\CT^+_l$ lies on its boundary.\label{txt:tilecorrection} As $(T_k)_{F_{d_k}}\supset(\widetilde{T_k})_{F_{d_k}}=({\hat{T}_k})_{H_{k-1}^2EF_{d_k}}$, we see that the union of $F_{d_k}$-cores of the tiles of $\CT_k$ has lower Banach density at least $(1-\eta_k)$. In addition, the condition $T_k\subset (\hat{T}_k)_{H_{k-1}^2}$ combined with the disjointness of $\hat{\CT}_k$ implies that for every $y$ the set $H_{k-1}$ has nonempty intersection with at most one tile of $\CT_k$.

Let $E_{k-1}=\bigcup_{l=1}^{k-1}S^*_l(S^*_l)^{-1}H_{l-1}^2$. According to lemma \ref{lem:disjoint_quasitiling} we can also assume that every tile $\hat T_k$ of $\hat{\CT}_k$ is $(E_{k-1},\eta_k)$-invariant, which makes the corresponding tile $T_k$ $(E_{k-1},2\eta_k)$-invariant. By an analogous reasoning, we can require that $(T_k)_{E_{k-1}^{-1}F_{d_k}}$ be a $(1-\eta_k)$-subset of $T_k$.

\begin{figure}
    \centering
    \includegraphics[width=6truecm]{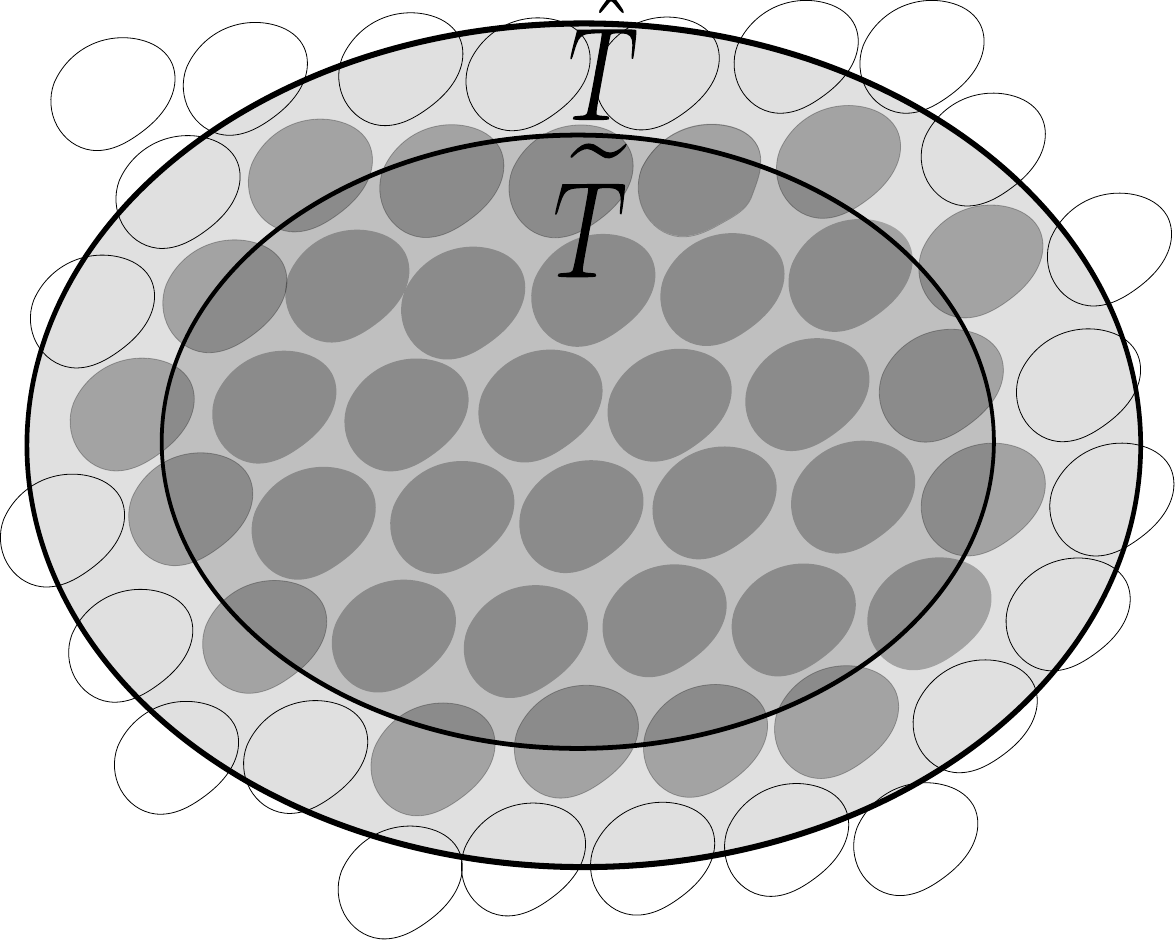}
    \includegraphics[width=6truecm]{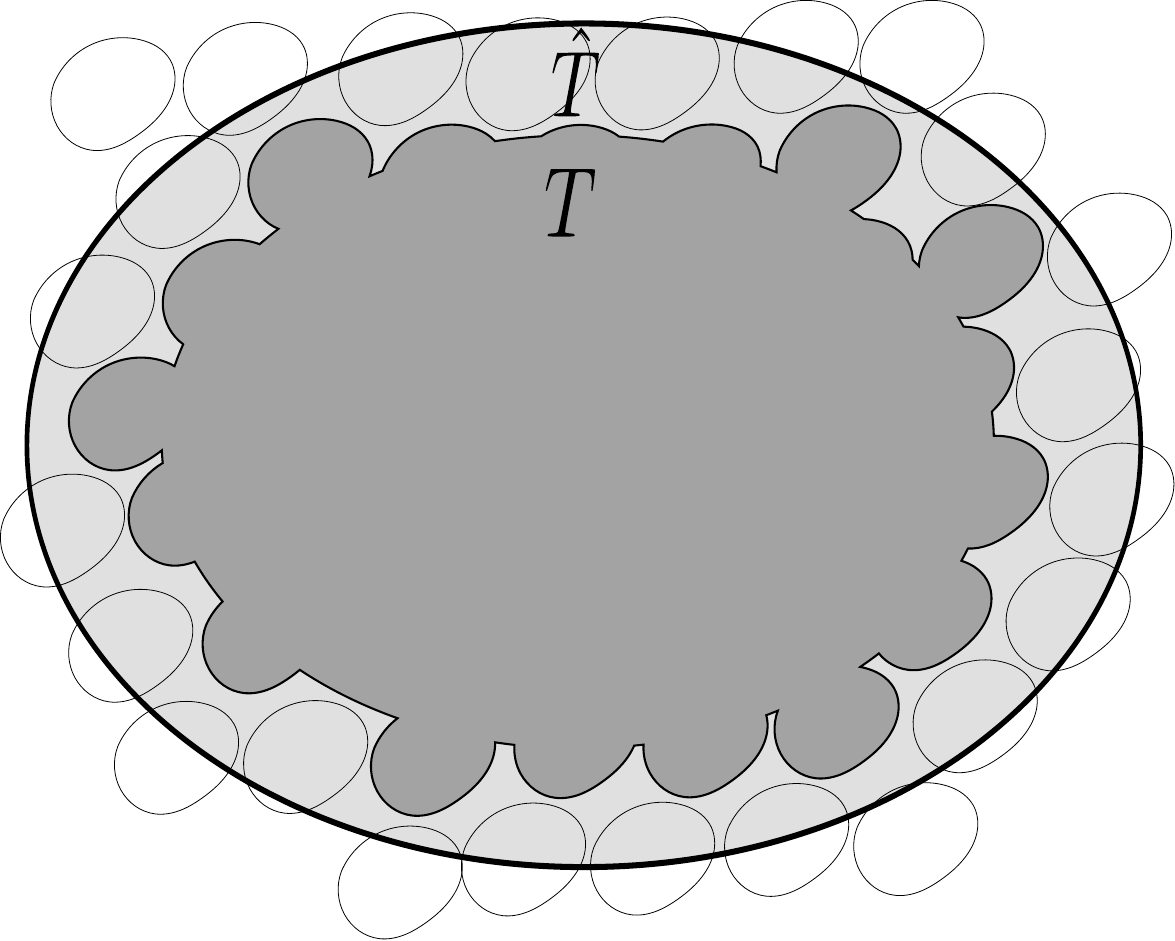}
    \caption{Construction of the tiles of $\CT_k(\cdot)$. The small sets are the tiles of $\CT^+_{k-1}$. The procedure of adding tiles from $\CT^+_l$ for $l<k-1$ is not depicted.}
    
\end{figure}

The next object we construct, the quasitiling $\CT^*_k$, will also be obtained from lemma \ref{lem:disjoint_quasitiling}. Before we apply the lemma, we must carefully choose its parameters.

For every $S\in\CS(\CT_k)$ the set of $y\in Y_{k-1}$ for which $Se\in \CT_k(y)$ is clopen (because $\CT_k$ is a dynamical quasitiling). Let $M_S$ be the image of this set by $\pi_{k-1}$. $M_S$ is the union of closures of finitely many atoms of $\bigvee_{g\in F_m}\Theta^{(X \times I)}(g^{-1})\CA_j$ for some $m$ and $j$. This implies that for every $\mu \in \CM_{\Theta^{(X)}}(X)$ the boundary of $M_S$ has measure $\mu \times \lambda$ equal to zero, therefore for some $t_\mu$ the horizontal section of $M_S$ at level $t_\mu$, i.e. the set $M_{S,t_\mu}=\set{x\in X:(x,t_\mu)\in M_S}$, has boundary of measure $\mu$ zero. Indeed, the set of such $t_\mu$ has Lebesgue measure $1$, so we can also choose $t_\mu$ so that it is generic for $\lambda$ with respect to the action $\Theta^{(I)}_\mu$ disjoint from $\mu$ and obtained from lemma \ref{lem:disjoint_group}. 
Let $M_{S,\mu}=M_{S,t_{\mu}}\times I$. Observe that since the quasitiling $\CT_k$ is disjoint and  $(1\!-\!\eta_k)$-covering, the images $\Theta^{(X \times I)}(g)(M_{S,\mu})$ for different $S\in \CS(\CT_k)$ and $g\in S$ are disjoint, and their total measure is $1-\eta_k$ for every $\Theta^{(X \times I)}$-invariant measure on $X\times I$. Furthermore, for every $S\in \CS(\CT_k)$ and every $g\in S$ the partition 
\[\CA'_{\mu,g}=\set{M_{S,\mu},(X\times I)\setminus M_{S,\mu}}\vee\set{\Theta^{(X \times I)}(g^{-1})(\pi_{k-1}(Q)):Q\in \CQ_k}\]
has boundaries of measure $\mu \times \lambda$ equal to zero. 

\begin{figure}
\centering
\includegraphics[width=6truecm]{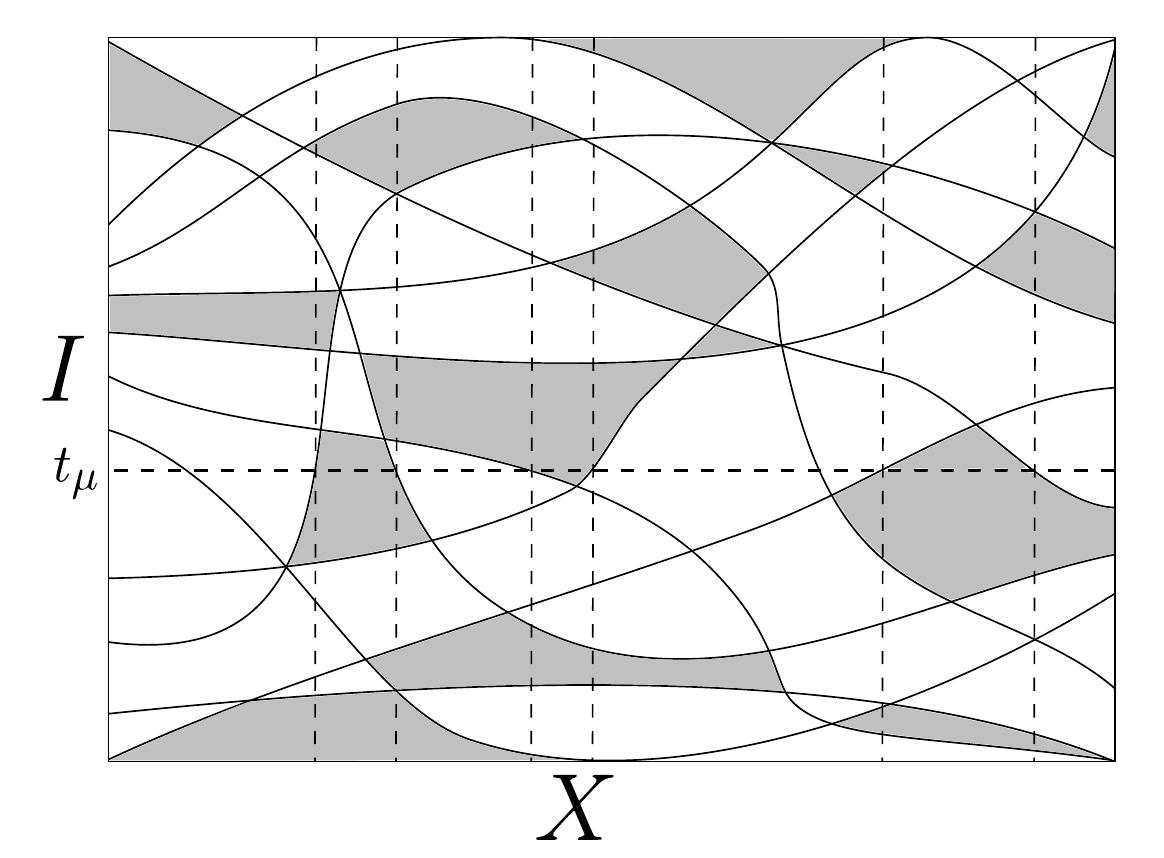}
\includegraphics[width=6truecm]{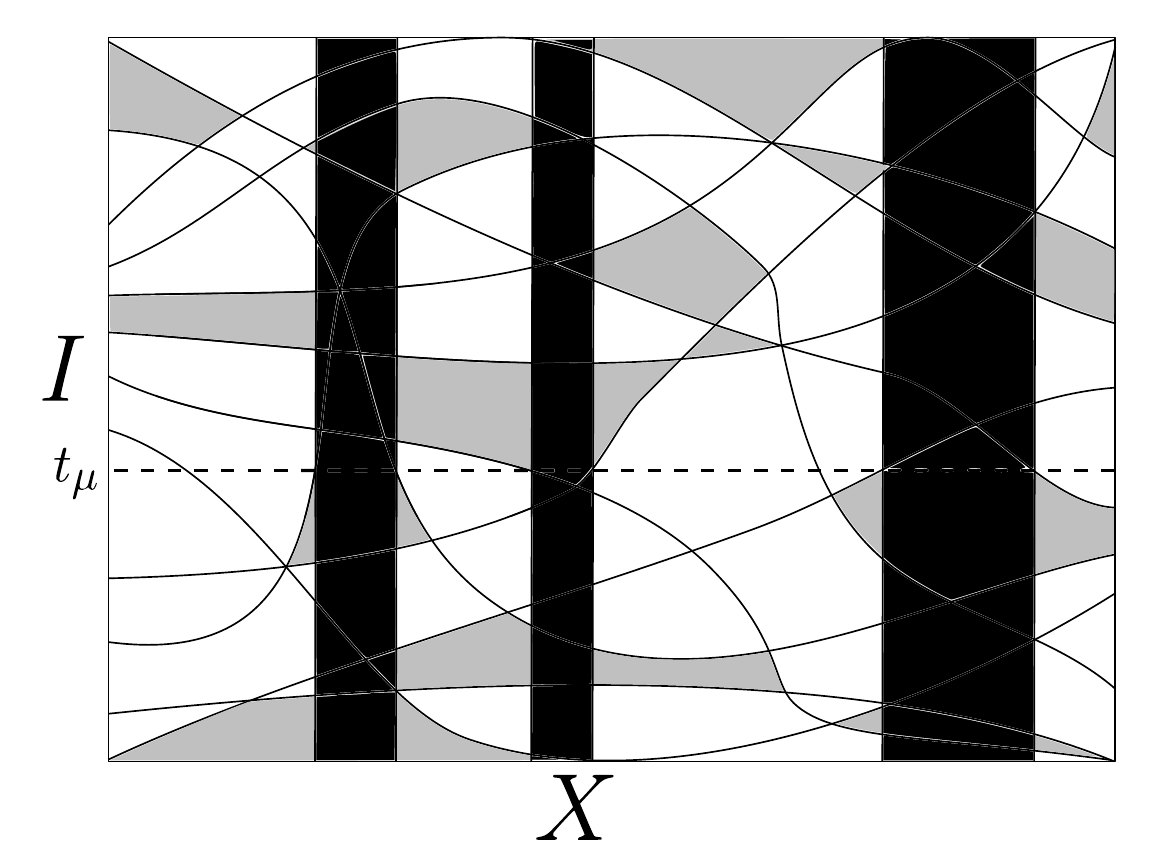}
\caption{Construction of the set $M_{S,\mu}$. The curves represent the graphs of the functions determining the partition $\bigvee_{g\in F_m}\Theta^{(X \times I)}(g^{-1})\CA_j$. The gray area in the left figure is the set $M_S$, while the black area in the right figure is the set $M_{S,\mu}$.}

\end{figure}

Therefore for every $A \in \CA'_{\mu,g}$ the function $\nu \mapsto \nu(A)$ is continuous at $\mu \times \lambda$, so there exists a neighborhood $U$ of $\mu \times \lambda$ in the space $\CM(X \times I)$ such that for every $\nu \in U$ we have
\[\abs{\nu(A)-(\mu\times\lambda)(A)}<\frac{\eta_k}{\abs{S}}\quad\text{for every }S\in \CS(\CT_k), g\in S_k\text{ and }A \in \CA'_{\mu,g}.\]

Apply lemma \ref{lem:products_group} for the measure $\mu$, neighborhood $U$ and action $\Theta_\mu^{(I)}$. We obtain a neighborhood $U_\mu$ of $\mu$ in $\CM(X)$, as well as a number $N_\mu$, such that for $n>N_\mu$ if  $\mathbf{A}^{\Theta^{(X)}}_{F_n}\delta_x\in U_\mu$, then $\mathbf{A}^{\Theta^{(X \times I)}_\mu}_{F_n}\delta_{(x,t_\mu)}\in U$. The $U_\mu$'s are an open cover of the compact set $\CM_{\Theta^{(X)}}(X)$, so there exists a finite family $\CW$ of $\Theta^{(X)}$-invariant measures on $X$, such that the union $\bigcup_{\mu \in \CW}U_\mu$ contains all of $\CM_{\Theta^{(X)}}(X)$. This union is an open set in $\CM(X)$, therefore there exists some $N$ such that for every $x \in X$ the average $\mathbf{A}^{\Theta^{(X)}}_{F_N}\delta_{x}$ is in $U_\mu$ for some $\mu \in \CW$. We can also assume that $N>N_\mu$ for all $\mu \in \CW$. 

Now apply lemma \ref{lem:disjoint_quasitiling} to $Y_{k-1}$, with the set $E_{k-1}S_k$ and the constant $\eta_k$, obtaining a quasitiling $\CT^*_k$. We can assume that every tile of $\CT^*_k$ is a $(1-\frac{\eta_k}{\abs{S_k}})$-subset of a \Fo set $F_N$ (see the remark at the end of the proof of lemma \ref{lem:disjoint_quasitiling}), and that all tiles of $\CT^*_k$ satisfy lemma \ref{lem:frequency} as the set $E$ (with $F_{d_k}$ as the set $A$ and $\eta_k$ as $\eps$), which means that the frequency of the ``test'' blocks in far \Fo sets differ by at most $\eta_k$ from their average frequencies in the tiles of $\CT^*_k$.

We will now modify the quasitiling $\CT_k$, leaving only the tiles which are entirely contained in tiles of $\CT^*_k$: for each $y\in Y_{k-1}$ replace $\CT_k(y)$ (without changing notation) with the set $\set{T\in \CT_k(y): \exists_{T^*\in\CT_k^*(y)}T_k\subset T_k^*}$. Observe that this does not affect the frequency with which the tiles of $\CT_k$ occur in the tiles of $\CT_k^*$.

We can now begin to construct the mapping $\Phi_k$, and thus also the system $Y_k$. Let $P$ be any block in $Y_{k-1}$ of height $j_k$, determined by the coordinates $S^*$ for some $S^*\in \CS(\CT^*_k)$. ($P$ is an analogue of the encoded $k$-rectangle in the proof of theorem 3.1 of \cite{DH2}). Fix some $y_P\in Y_{k-1}$, so that $e$ is the center of some tile $T^*$ of $\CT^*_k$ and $P=y_P(S^*)$ where $S^*$ is the shape of $T^*$. The image of $P$ by $\pi_{k-1}$ in $X \times I$ is the closure of a union of some atoms of $\CA_{j_k}^{F_N}$. Let $x_P=\pi_X(y_P)$ and observe that $x_P$ belongs to the projection of one of these atoms onto $X$. 

We should also make note of the following fact which we will need towards the end of the proof: $N$ is so large that the average $\mathbf{A}^{\Theta^{(X)}}_{F_N}\delta_{x_P}$ is in the set $U_{\mu_P}$ for some $\mu_P \in \CW$, therefore (by lemma \ref{lem:products_group} and our choice of constants)
\begin{equation}\label{eq:products_group}\abs{\mathbf{A}^{\Theta^{(X \times I)}_{\mu_P}}_{F_N}\delta_{(x_P,t_{\mu_P})}(A)-(\mu_P\times\lambda)(A)}<\frac{\eta_k}{\abs{S_k}}\quad\text{for all }A \in \CA'_{\mu,g}\end{equation}
(for every $g\in S_k$). 

For every coordinate $b$ of $P$ such that $b$ is a center of $\CT_k$, let $y_b$ be any preimage by $\pi_{k-1}$ of $(x_P,\Theta^{(I)}_{\mu_P}(b)t_{\mu_P})$ (the $y_b$'s are an analogue of the $y_i$'s in the proof of theorem 3.1 of \cite{DH2}). 


For each $l<k$ define two sets $\hat{\CT}_l(P),\hat{\CT^*}_l(P)\subset G$ (we will later use those to define the dynamical quasitilings $\CT_l,\CT_l^*$ for $Y_k$). Initially let $\hat{\CT}_l(P)$ and $\hat{\CT^*}_l(P)$ consist of the tiles of $\CT_l(y_P)$ and $\CT^*_l(y_{P})$ (respectively) which have nonempty intersection with $S^*$ .


We can now describe the construction of $\Phi_k(P)$. For every $b\in \CC(\CT_k)$ such that the tile $Sb$ with centre $b$ is such that $E_{k-1}S$ is entirely contained in  $S^*$, do the following:
\begin{enumerate}
    \item Replace in $P$ all symbols in rows $1$ to $j_{k}$ at coordinates $Sb$ with the symbols that appear in $y_b$ at these coordinates. 
    
    Additionally, for $l<k$ remove from $\hat{\CT_l}(P)$ all the tiles of  $(\CT_l(y_P))$ which are subsets of $Sb$ and add to it all tiles of $\CT_l(y_b)$ which intersect $Sb$. Modify $\hat{\CT^*}_l(P)$ in the analogous manner. The interpretation of this is that within the tile $S$ we transfer from $y_b$ not only the symbols, but also the tiles of $\CT_l$ and $\CT^*_l$.
    
    \item Let $l=k-1$. For every $g$ for which $H_{l-1}g$ lies on the boundary of $Sb$, in the ,,area'' (this notion will be specified below) of $H_{l-1}g$ we undo the changes made in step $l$ (within $Sb$ we restore the previous content of $y_b$, and outside $Sb$ we restore the previous content of $y_P$). 
    
    The precise modification is as follows: For every tile $T^*_l$ of $\CT_l*(y_b)$ which has nonempty intersection with $H_{l-1}g$, replace the contents of coordinates from $T^*_l\cap Sb$ in rows from $1$ to $j_{l}$ with the contents of the corresponding coordinates $T^*_l\cap Sb$ in $\Phi_{l}^{-1}(y_b)$. Also, remove $T^*_l$ from $\hat{\CT^*}_l(P)$ and remove from $\hat{\CT_l}(P)$ all tiles which are contained in $T^*_l$.
    
    Now do the same for $y_P$ rather than $y_b$: for every tile $T^*_l$ of $\CT^*_l(y_P)$ which has nonempty intersection with $H_{l-1}g$, replace the contents of coordinates $T^*_l \setminus Sb$ in rows from $1$ to $j_{l}$ with the contents of the corresponding coordinates  $T^*_l\setminus Sb$ in $\Phi_{l}^{-1}(y_P)$. Also, remove $T^*_l$ from $\hat{\CT^*}_l(P)$, and remove from $\hat{\CT}_l(P)$ all tiles contained in $T^*_l$.
    
    Observe that after such a modification all the blocks in $\Phi_k(P)$ determined by the coordinates $H_{l-1}g\cap Sb$ i $H_{l-1}g\setminus Sb$ are blocks that occur in $Y_{l-1}$.\footnote{The role of item 2 is the same as of conditions on the placement of jump points in the proof of theorem 3.1 of \cite{DH2} -- our modification ensures that the boundaries of tiles that are ``replaced'' in subsequent steps of the induction do not accumulate, which would create regions in which we would not control the frequency of ``test'' blocks.} Observe also that $E_{l}$ was chosen so that the coordinates affected by the modifications of this step are all in the set $(E_{l}Sb\setminus Sb)\cup (E_{l}(G\setminus Sb)\cap Sb)$.
    
    Now do the same for $l=k-2,k-3,\ldots,1$. 
\end{enumerate}
\indent The other symbols of $P$ are left unchanged. The inclusion $E_{k-1}Sb\subset S^*$ ensures that the modifications performed in items 2 and 3 affect only coordinates from $S^*$, i.e. the domain of $P$. 

\begin{figure}
    \centering
    \includegraphics[width=11truecm]{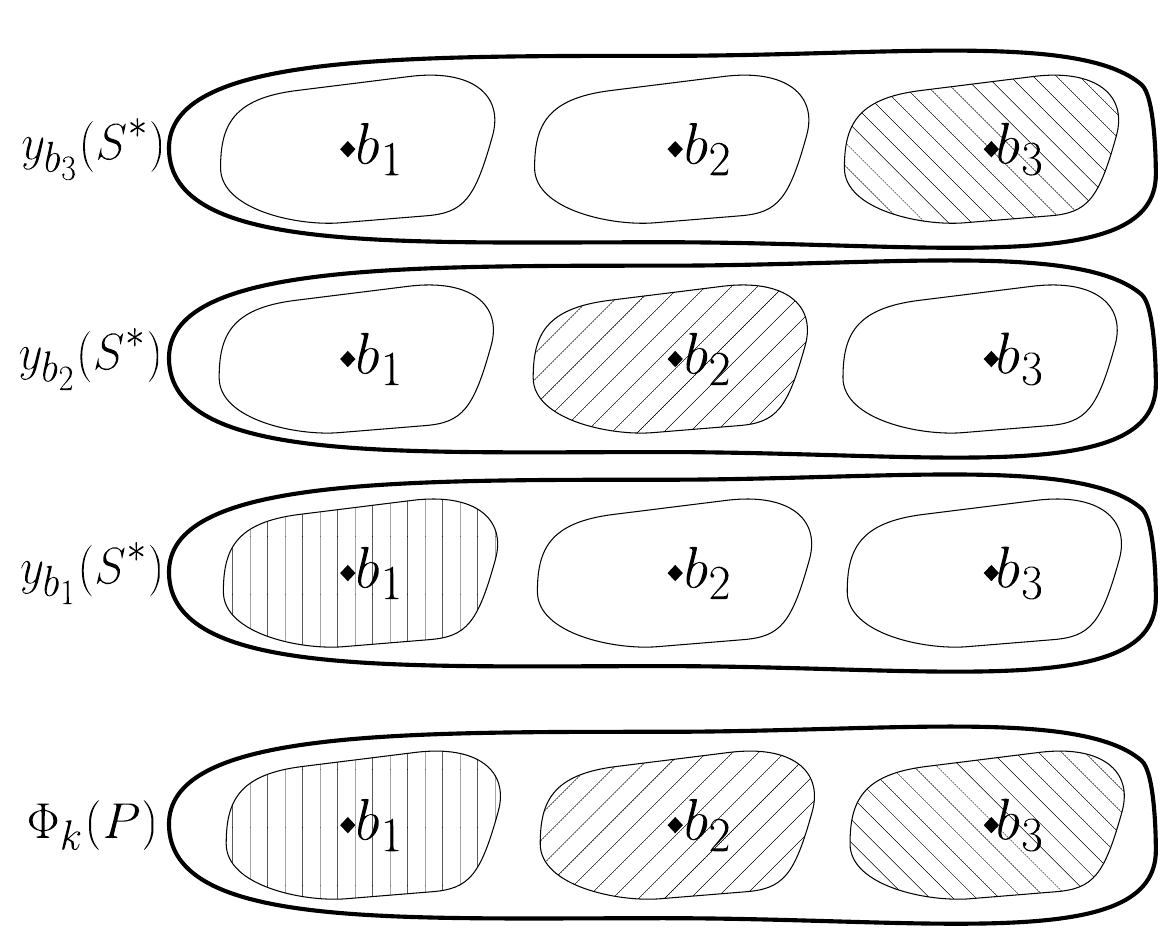}
    \caption{The first step of the code $\Phi_k$ The larger set is $S^*$, and the smaller sets are the tiles of $\CT_k(y_P)$ for $i=1,2,3$, $b_i\in B_k(y_P)$ -- in every such tile $T=S_ib_i$ the block $\Phi_k(P)$ has the contents of the coordinates from $T$ in $y_{b_i}$. For simplicity we show all tiles as having the same shape, but this does not affect the actual construction.}
\end{figure}

\begin{figure}
    \centering
    \includegraphics[width=9truecm]{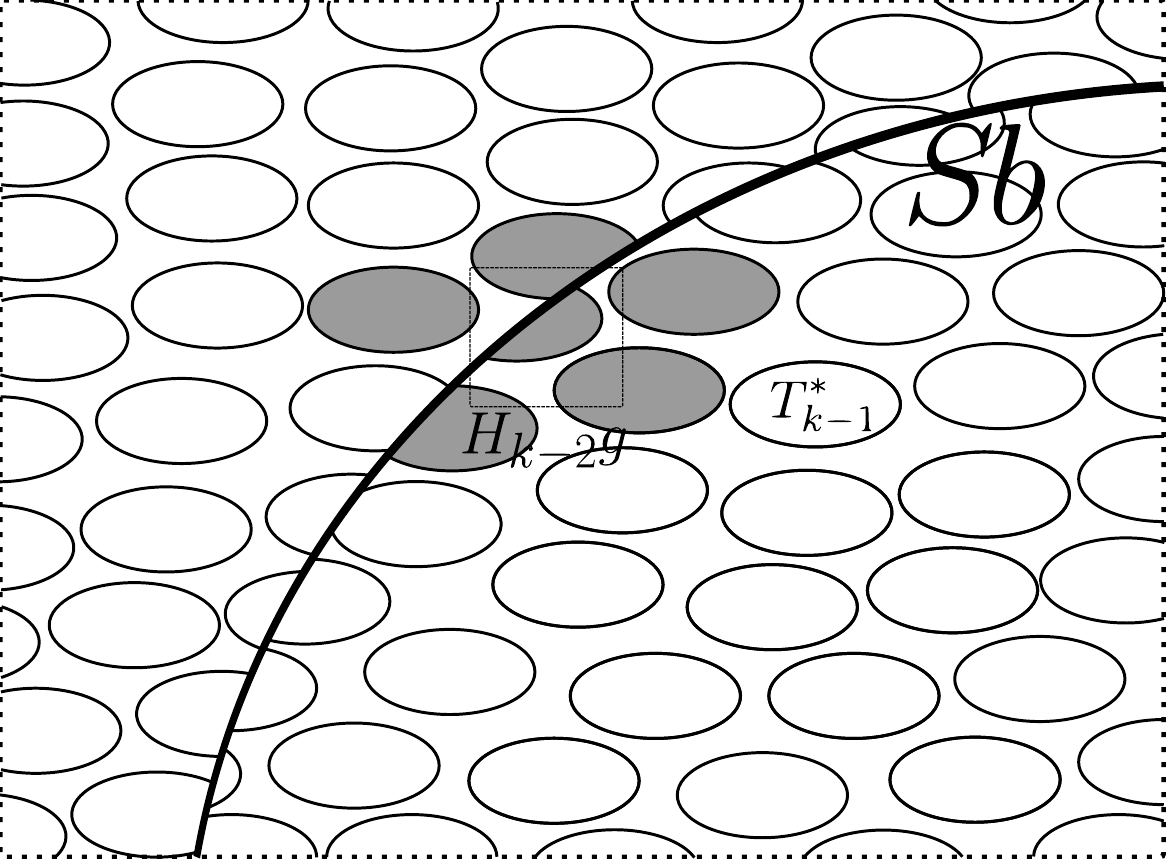}
    \medskip
    
    \includegraphics[width=9truecm]{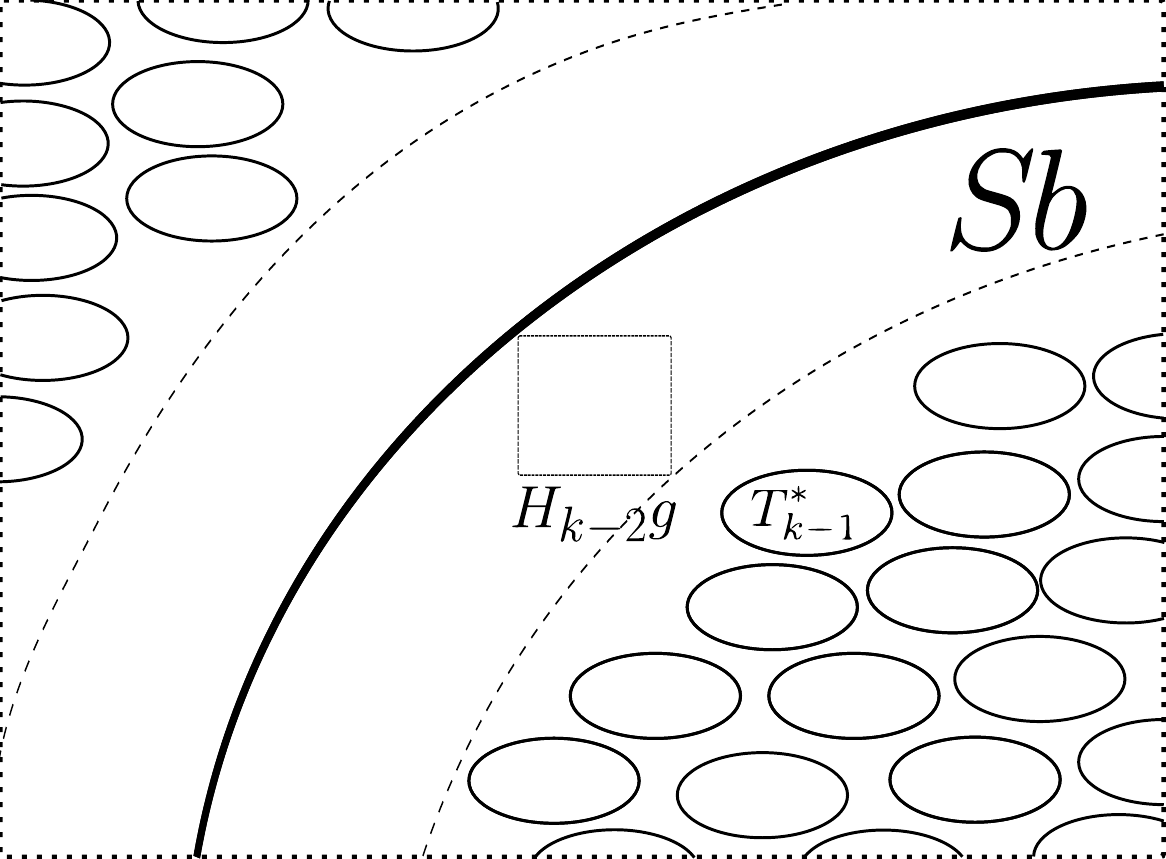}
    \caption{The second step of the code $\Phi_k$ for $l=k-1$. The small square is the set $H_{k-2}g$ that lies on the boundary of $Sb$, whereas the smaller sets are the tiles $T^*_{k-1}$. In every such tile that has nonempty intersection with $H_{k-2}g$ we restore the contents of the corresponding coordinates of $\Phi^{-1}_{k-1}(y_b)$ (within $Sb$) or of $\Phi^{-1}_{k-1}(y_P)$ (outside $Sb$). Observe that any coordinates that are not in some tile of $\CT^*_{k-1}$ have not been changed by $\Phi_{k-1}$, so both parts of $H_{k-2}g$ (the intersection with $Sb$ and with its complement) are blocks that occur in $Y_{k-2}$. The bottom figure depicts $\Phi_k(P)$ after restoring for all $g$ the former content of all the tiles of $\CT^*_{k-1}$ that have nonempty intersection with the set $H_{k-2}g$ lying on the boundary of $Sb$.}
\end{figure}

This block code induces a map on $Y_{k-1}$ defined as follows: For every $y\in Y_{k-1}$ let $\Phi_k(y)$ be the point obtained by replacing, for every tile $S^*b\in \CT^*_k(y)$, the symbols in the tile $y(S^*b)$ with the corresponding symbols from $\Phi_k(y(S^*b))$. The coordinates which are not covered by any tile of $\CT^*_k(y)$ remain unchanged. Since $\CT^*_k$ is a dynamical quasitiling, the map $\Phi_k$ is continuous and commutes with the group action. As $\Phi_k$ makes no changes in rows greater than $j_k$, it is a conjugacy.

Let $Y_k=\Phi_k(Y_{k-1})$. It is easy to see that the sets $Y_k$ and $Y_{k-1}$ are disjoint (no block of the form $\Phi_k(P)$ occurs in $Y_{k-1}$). The quasitilings $\CT_k$ and $\CT_k^*$ were defined on $Y_{k-1}$, so we can transfer them onto $Y_k$ by images, i.e. for every $y\in Y_k$ let $\CT_k(y)=\CT_k(\Phi_k^{-1}(y))$ and $\CT^*_k(y)=\CT^*_k(\Phi_{k}^{-1}(y))$. We will also transfer onto $Y_k$ all the quasitilings $\CT_l$ and $\CT^*_l$ for $l<k$: for $y\in Y_k$ and every tile $T^*$ of $\CT^*_k$, include in $\CT^*_l(y)$ all the elements of $\hat{\CT^*}_l(P)$, where $P$ is the block occurring within the same tile in $\Phi_k^{-1}(y)$ (translating them to the appropriate relative position). In addition, include in $\CT^*_l(y)$ all the tiles of $\CT^*_l(\Phi_{k}^{-1}(y))$ which are not contained within tiles of $\CT^*_k(y)$. The set $B^*_l(y)$ is constructed analogously.

This ends the inductive construction of $Y_k$.
\medskip
We must show that every invariant measure on $Y_k$ is in $\CU_k$. For this it suffices to show that for any ergodic measure $\nu$ on $Y_k$ there exists some $\nu' \in \CP_k$ such that the two measures differ by no more than $\delta_k$ on all cylinders from $\CQ_k$.

Let $P$ be the block determined by coordinates from $S^*$ in $y_P$, as above. We will show that every $Q \in \CQ_k$ occurs in $\Phi_k(P)$ with frequency close to $\nu_P(Q)$, where $\nu_P$ is the only measure on $Y_{k-1}$ whose image by $\pi_{k-1}$ is $\mu_P \times \lambda$. Indeed, let $(x_P,t_{\mu_P})$ be the point chosen when constructing $\Phi_k(P)$. Observe that the $F_{d_k}$-cores of the tiles $T\in \CT_k(y_P)$ contained in $S^*$ cover at least $1-4\eta_k$ of $S^*$ (we use the construction of $S^*$: the sets of the form $E_{k-1}T$, as $T$ ranges over the tiles of $\CT_k$, cover $1-\eta_k$ of $S^*$, and then we can use the invariance of $S$ with respect to $E_{k-1}$ and $F_{d_k}$). Also observe that if $b$ is a centre of a tile $Sb$ of $\CT_k(y_P)$ and $g\in (S)_{E_{k-1}^{-1}F_{d_k}}$, then the block $Q \in \CQ_k$ occurs in $\Phi_k(y_P)$ at coordinate $gb$ if and only if $\Theta^{(X \times I)}_{\mu_P}(b)(x_P,t_{\mu_P})$ is in the set $M_{S,\mu_P}\cap \Theta^{(X \times I)}(g^{-1})\pi_{k-1}(Q)\in \CA'_{\mu_P,g}$. The expression

\[\begin{split}
\frac{1}{\abs{S^*}}\sum_{S\in \CS(\CT_k(y_P))}\bigg(\sum_{b\in \CC(S)\cap (S^*)_{E_{k-1}S_k}}\bigg(\sum_{g \in (S)_{E_{k-1}^{-1}F_{d_k}}}\\
1_{M_{S,\mu_P}\cap \Theta^{(X \times I)}(g^{-1})\pi_{k-1}(Q) }(\Theta^{(X \times I)}_{\mu_P}(b)(x_P,t_{\mu_P}))\bigg)\bigg)
\end{split}\]

is (apart from the division by $\abs{S^*}$) the total number  of occurrences of $Q$ in $\Phi_k(P)$ within the tiles which were modified in the first step of $\Phi_k$, but not in the second, i.e. at coordinates of the form $F_{d_k}gb$ where $T=Sb\in \CT_k(y_P)$, $E_{k-1}T\subset S^*$, and $g\in (S)_{E_{k-1}^{-1}F_{d_k}}$. Such $gb$'s are at least $1-4\eta_k$ of all the elements of $S^*$, so the above expression differs from $\fr_{\Phi_k(P)}(Q)$ by at most $4\eta_k$. For the same reason we can replace cores with whole sets in both sums without changing the value of the expression by more than $4\eta_k$. This modification gives us the expression

\[\frac{1}{\abs{S^*}}\sum_{S\in\CS(\CT_k(y_P))}\sum_{b\in \CC(S)\cap S^*}\sum_{g\in S}
1_{M_{S,\mu_P}\cap \Theta^{(X \times I)}(g^{-1})\pi_{k-1}(Q) }(\Theta^{(X \times I)}_{\mu_P}(b)(x_P,t_{\mu_P})).\]


Furthermore, if $b$ is not a center of any tile of $\CT_k(y_P)$, the point $\Theta^{(X \times I)}_{\mu_P}(b)(x_P,t_{\mu_P})$ does not belong to $M_{S,\mu_P}$, since $(x_P,t_{\mu_P})$ visits $M_{S,\mu_P}$ at the same times under the action of $\Theta^{(X \times I)}_{\mu_P}$ as under the action of $\Theta^{(X \times I)}$ (since these visits only depend on the first coordinate), and visits of $(x_P,t_{\mu_P})$ under $\Theta^{(X \times I)}$ in $M_{S,\mu_P}$ are equivalent to visits of $y_P$ in $\set{y:e\in \CC(S)}$. Therefore we can actually sum over all $b\in S^*$ without changing the value of the sum. This gives us the following estimate, where the symbol $\overset{8\eta_k}{\approx}$ and similar indicate that the absolute difference between both sides of the symbol does not exceed the stated number. The second approximation is a consequence of estimate \ref{eq:products_group} and the fact that $S^*$ is a large subset of $F_N$, as stated two paragraphs before said estimate): 

\[\begin{split}\fr_{\Phi_k(P)}(Q)
\overset{8\eta_k}{\approx}\sum_{S\in\CS(\CT_k(y_P))}\sum_{b\in S^*}\sum_{g\in S}1_{M_{S,\mu_P}\cap \Theta^{(X \times I)}(g^{-1})\pi_{k-1}(Q) }(\Theta^{(X \times I)}_{\mu_P}(b)(x_P,t_{\mu_P}))=\\
=\sum_{S\in\CS(\CT_k(y_P))}\sum_{g \in S}\left(\frac{1}{\abs{S^*}}\sum_{b\in S^*}1_{M_{S,\mu_P} \cap \Theta^{(X \times I)}(g^{-1})\pi_{k-1}(Q) }(\Theta^{(X \times I)}_{\mu_P}(b)(x_P,t_{\mu_P})\right)\overset{2\eta_k}{\approx}\\
\overset{2\eta_k}{\approx}\sum_{S\in\CS(\CT_k(y_P))}\sum_{g \in S}(\mu_P \times \lambda) (M_{S,\mu_P} \cap \Theta^{(X \times I)}(g^{-1})\pi_{k-1}(Q))=\\
=\sum_{S\in\CS(\CT_k(y_P))}\sum_{g \in S}(\mu_P \times \lambda) (\Theta^{(X \times I)}(g)(M_{S,\mu_P})\cap \pi_{k-1}(Q) )\overset{*}{=}\\
\overset{*}{=}(\mu_P \times \lambda) \left(\left(\bigcup_{S\in\CS(\CT_k(y_P))}\bigcup_{g \in S}\Theta^{(X \times I)}(g)(M_{S,\mu_P})\right)\cap \pi_{k-1}(Q) \right)\overset{\eta_k}{\approx}\\
\overset{\eta_k}{\approx} (\mu_P \times \lambda)(\pi_{k-1}(Q)),\end{split}\]

where the equality ($*$) and the following estimate are true, because the images of $M_{S,\mu_P}$ by the elements $g\in S, S\in\CS(\CT_k)$ are disjoint and their total measure $\mu_P\times \lambda$ is at least $(1-\eta_k)$. Therefore we have shown that for every $Q\in \CQ_k$
\begin{equation}\abs{\fr_{\Phi_k(P)}(Q)-(\mu_P \times \lambda)(\pi_{k-1}(Q))}<11\eta_k.\end{equation}

Now, let $y\in Y_k$ be generic for $\nu\in \CM_{\Theta^{(Y)}}(Y_k)$. By the properties of $\CT^*_k$ determined two paragraphs before estimate \ref{eq:products_group}, we conclude that the frequency of $Q$ in $y$ (equal to $\nu(Q)$) differs by less than $12\eta_k$ from the value on the set $Q$ of some convex combination (which does not depend on $Q$) of measures of the form $\pi_{k-1}^{-1}(\mu_{P}\times \lambda)$, and such a combination is in $\CP_k$. As long as $\eta_k<\frac{\delta_k}{12}$ (which we can assume), we have $\nu\in \CU_k$. This is equivalent to stating that if $F_n$ is a sufficiently large \Fo set, then for every $y\in Y_k$ there exists a measure $\nu_y \in \CP_k$, which satisfies the following inequality for every $Q\in \CQ_k$: 
\begin{equation}\label{eq:folner_frequency}\abs{\fr_{y(F_n)}(Q)-\nu_y(\pi_{k-1}(Q))}<12\eta_k.\end{equation}

\medskip

We can now define the quasitiling $\CT^+_k$. Let $E'=\bigcup_{l=1}^{k-1}(T_l^+)^2$, where $T_l$ is the union of all the shapes of $\CT_l^+$. Apply lemma \ref{lem:disjoint_quasitiling}, obtaining a disjoint quasitiling $\hat{\CT}_k^+$ whose tiles are $(F_{d_k},\eta_k)$-invariant and corresponding quasitiling by the $E'E^{-1}_{k}$-cores of the tiles is $(1\!-\!\eta_k)$-covering. Let $\CT^+_k=\set{\hat{T}_{E'}:\hat{T}\in\hat{\CT_k^+}}$. Note that if $T\in \CT^+_k$, and $T=\hat{T}_{E'}$ for some $\hat T\in \hat{\CT_k^+}$, then $(T)_{E_{k}^{-1}}=(\hat T)_{E'E^{-1}_{k}}$ is a $(1-\eta_k)$-subset of $\hat{T}$ (as otherwise $((\hat T)_{E'E^{-1}_{k}},C_k)$ would not be a $(1\!-\!\eta_k)$-covering quasitiling), so it is $(F_{d_k},2\eta_k)$-invariant. Furthermore, the definition of $\CT^+_k$ ensures that the first inductive assumption is fulfilled (for $k$): any set of the form $T_lg$ for $l<k,g\in G$ has nonempty intersection with at most one tile of $(T_k,C_k)$.

The final step of the inductive construction is identifying the set $H_k$. This will be an element of the \Fo sequence which is sufficiently large that it satisfies the statement of lemma \ref{lem:small_boundary} for $\CT^+_k$ with  $\eta_k$ in the role of $\eps$. It must also be large enough to ``realize'' the lower Banach density (equal to $1-\eta_k$) of the union of $\CT^+_l$. Then for every $g\in G$ the union of tiles of $\CT^+_l$ that are contained within $H_kg$ is a $(1-2\eta_k)$-subset of $H_kg$. \medskip

It remains to verify the second inductive condition which requires that $Y_k$ have the property Y3($l$) for all $l<k$. Fix $l$ and $y\in Y_k$. The properties of the shapes of $\CT_{l+1}(y)$ (see page \pageref{txt:tilecorrection}) imply that $H_{l}$ lies at the boundary of at most one tile $\CT_{l+1}(y)$. Similarly, it lies on the boundary of at most one tile of any $\CT_{l'}(y)$ for $l'>l$, since this is a property of $H_{l'-1}$ which contains $H_l$ (recall that $H_l$ are unmodified \Fo sets). Finally $H_l$ does not lie simultaneously on the boundaries of two tiles of $\CT_{l'}(y)$ and $\CT_{l''}(y)$ for some $l'\neq l''$ --- this is ensured by the modifications in step 2 of the code construction. Ultimately we see that $H_l$ lies on the boundary of at most one tile of $\CT_{k'}(y)$ where $l<k'\leq k$.

$H_{l}$ is covered in proportion $1-2\eta_{l}$ by tiles of $\CT^+_l$ which are its subsets. Each of these tiles is either a subset of some tile of $\CT_{k'}(y)$ or is disjoint from all such tiles (since $S_{k'}$ was constructed so that no tile of $\CT_l^+$ lies on its boundary -- see page \pageref{txt:tilecorrection}). It follows that every such tile determines a block that occurs in $Y_{l}$, up to coordinates which were modified in step 2 of $\Phi_{k'}$. Let $g$ be one of such coordinates in a tile $T\in\CT^+_l$ (of the form $Sc$) which is disjoint with all tiles of $\CT_{k'}$. (the case where $T$ is a subset of some tile of $\CT_{k'}$ is analogous). Since $g$ was modified, it must belong to $E_{l}T'$ for some $T'\in \CT_{k'}(y)$, so $E_{l}^{-1}g$ has nonempty intersection with $T'$, which in particular implies that it is not a subset of $T$, therefore $g\notin (T)_{E_{l}^{-1}}$. As $(T)_{E_{l}^{-1}}$ is a $(1-\eta_{l})$-subset of $T$, we conclude that the modified coordinates form a small part of $T$. The union of the $E_l^{-1}$-cores of tiles of $\CT^+_l$ contained in $H_{l}$ is a $(1-3\eta_{l})$-subset of $H_{l}$. Furthermore, any such $E_{l}^{-1}$-core is $(F_{d_{l}},\eta_{l})$-invariant, so if $\eta_{l}$ is sufficiently small, the estimate \ref{eq:folner_frequency} (for $l$) and lemma \ref{lem:frequency} give us a measure $\nu_y\in \CP_l$ such that for every $Q\in \CQ_{l}$ we have the following inequality
\begin{equation}
\label{eq:hk_frequency}\abs{\fr_{y(H_{l})}(Q)-\nu_y(\pi_{l}(Q))}<\delta_l,
\end{equation}
which is required by condition Y3($l$).

\medskip
We have constructed a set $\CU_k$ with properties U1-U3, a system $Y_k$ with properties Y1, Y2 and Y3($l$) for all $l<k$, as well as quasitilings $\CT_k(\cdot)$, $\CT^*_k(\cdot))$ and $\CT^+_k$, and the set $H_k$, so the inductive step is complete.
\medskip

Now define the final system

\[Y=\bigcap_{m=1}^\infty\overline{\bigcup_{k=m}^{\infty}Y_k}.\] 

Observe that $Y$ has property Y3($l$) for all $l$. Indeed, let $y\in Y$. There exists some $k>l$ and $y_k\in Y_k$, so that $y(H_l)=y_k(H_l)$. Since $Y_k$ has property Y3($l$), this block satisfies the required condition.

The fact that $Y$ has property Y3($l$) implies that all invariant measures on $Y$ are in $\CU_l$. Indeed, let $y$ be generic for some ergodic $\nu \in \CM_{\Theta^{(Y)}}(Y)$, let $F$ be a far \Fo set and let $Q\in \CQ_l$. Adding the frequencies of $Q$ in $y(H_lg)$ for $H_lg\subset F$ and dividing this sum by $g$, we obtain on one hand a convex combination of the frequencies of $Q$ in the blocks $gy(H_l)$ (a number different by at most $\delta_l$ from the same convex combination of measures $\nu_{gy,l} \in \CP_l$), and on the other -- by a standard argument involving changing the order of summation -- this number will be between $\fr_{y(F_{H_l})}(Q)$ and $\frac{\abs{F}}{\abs{F_{H_l}}}\fr_{y(F)}(Q)$. Ultimately we see that $\nu(Q)$ differs by less than $\delta_l$ from some convex combination of measures in $\CP_l$, and thus $\nu\in \CU_l$.

\medskip

To show that $Y$ is a principal extension of $X$ we need to show that the conditional entropy of $Y$ with respect to $X$ is $0$ for every measure $\nu \in \CM_{\Theta^{(Y)}}(Y)$. For any $k>0$ and for any $k'>k$ we have $h(\nu,\mathcal{R}_{k}|X) \leq h(\nu,\mathcal{R}_{k'}|X)$, since $\mathcal{R}_{k'} \succ \mathcal{R}_{k}$. On the other hand, since $\nu$ is in the set $\CU_{k'}$, using the property U2, we know that $h(\nu,\mathcal{R}_{k'}|X)<\eps_{k'}$. It follows that for any $k'>k$ $h(\nu,\mathcal{R}_{k}|X)<\eps_{k'}$, and thus $h(\nu,\mathcal{R}_{k}|X)=0$. Thus we conclude that $h(\nu|X)=0$.

Similarly, since $\CM_{\Theta^{(Y)}}(Y) \subset \CU_k$ for every $k$, using the property U3, if two invariant measures on $Y$ factor onto the same measure on $X$, then they must be closer to each other than $\eps_k$ for all $k$, and thus every invariant measure on $X$ has exactly one preimage on $Y$. 
\end{proof}
\begin{proof}[Proof of theorem \ref{thm:principal_group}]
    By theorem 6.1 in \cite{DHZ}, for every amenable group $G$ there exists a zero-dimensional dynamical system $(Z,\Theta^{(Z)})$ with entropy zero, such that $\Theta^{(Z)}$ is a free action of $G$ on $Z$. It follows that the system $(X \times Z,\Theta^{(X)}\times \Theta^{(Z)})$ is also a free dynamical system, and as $(Z,\Theta^{(Z)})$ has entropy zero, this product is a principal extension of $(X,\Theta^{(X)})$. We can now apply theorem \ref{thm:faithful_group} to $(X \times Z,\Theta^{(X)}\times \Theta^{(Z)})$, obtaining its principal extension $(Y,\Theta^{(Y)})$. Obviously, the new system is also a principal extension of $(X,\Theta^{(X)})$. It is also a faithful extension of $(X \times Z,\Theta^{(X)}\times \Theta^{(Z)})$, but since the latter can have multiple preimages of some invariant measures on $(X,\Theta^{(X)})$, $(Y,\Theta^{(Y)})$ is not necessarily a faithful extension of $(X,\Theta^{(X)})$.
\end{proof}

\end{document}